\documentclass[10pt,a4paper,reqno,twoside]{amsart}

 \usepackage{latexsym}
 \usepackage{amsbsy}
\usepackage{amsmath,amssymb,amsthm}
\usepackage[english]{babel}
\usepackage[latin1]{inputenc}
\usepackage[T1]{fontenc}
\usepackage{enumerate}
\usepackage{mathtools}
\usepackage{color}
\usepackage{fancyhdr}
\usepackage{enumitem}

\usepackage{mathabx}
\usepackage{mathrsfs}

\def\@begintheorem#1#2{\bgroup{\sc #1 \ #2.}\it \ignorespace}
\def\@opargbegintheorem#1#2#3{ \bgroup{\sc #1\ #2 \ (#3).\\}\it \ignorespace}
\def\@endtheorem{\egroup}
\def\R{{\mathbb{R}}}
\def\N{{\mathbb{N}}}
\def\Z{{\mathbb{Z}}}
\def\C{{\mathbb{C}}}

\usepackage{color}
\usepackage[dvipsnames]{xcolor}

\newcommand{\<}{\langle}
\renewcommand{\>}{\rangle}
\renewcommand{\phi}{\varphi}

\newcommand{\supp}{\operatorname{supp}}
\newcommand{\loc}{\operatorname{loc}}
\newcommand{\absa}{\lambda_\mathrm{a}}

\newcommand{\Op}{\operatorname{Op}}

\newcommand{\funLT}{L^1_\cL}
\newcommand{\funLTR}{L^1_\cL(\R)}
\newcommand{\disLT}{\mathscr{D}^\prime_\cL}
\newcommand{\disLTR}{\mathscr{D}^\prime_\cL(\R)}

\newcommand{\norm}[1]{\langle#1\rangle}

\def\cA{{\mathcal A}}

\def\cL{{\mathcal L}}

\def\scD{{\mathscr D}}
\def\scE{{\mathscr E}}
\def\scF{{\mathscr F}}

\def\scS{{\mathscr S}}
\def\scSp{{\mathscr S}^\prime}

\def\jap#1{\langle {#1} \rangle}
\def\norm#1{\langle #1 \rangle}

\newcommand*{\scrL}{\ensuremath{\mathscr{L}}}	%linear continuous maps

\newcommand*{\ii}{\mathrm{i}}	
\newcommand{\x}{\langle x\rangle}
\newcommand{\csi}{\langle \xi \rangle}
\newcommand{\pdd}{\langle D \rangle}
\newcommand{\japdot}{\langle \cdot\rangle}
\newcommand*{\caO}{\ensuremath{\mathcal{O}}}

\newtheorem{thm}{Theorem}[section]
\newtheorem{prop}[thm]{Proposition}%[chapter]
\newtheorem{lem}[thm]{Lemma}%[chapter]
\newtheorem{cor}[thm]{Corollary}%[chapter]

\theoremstyle{definition}
\newtheorem{defn}[thm]{Definition}%[chapter]

\theoremstyle{remark}
\newtheorem{rem}[thm]{Remark}%[chapter]
\renewcommand{\theta}{\vartheta}

\newcommand{\beq}{\begin{eqnarray}}
\newcommand{\eeq}{\end{eqnarray}}
\newcommand{\beqst}{\begin{eqnarray*}}
\newcommand{\eeqst}{\end{eqnarray*}}
\newcommand{\be}{\begin{equation}}
\newcommand{\ee}{\end{equation}}

\DeclareMathOperator*{\vspan}{span}
\newcommand{\wtP}{\widetilde{P}}
\newcommand{\wtQ}{\widetilde{Q}}
\newcommand{\wte}{\widetilde{e}}
\newcommand{\wtth}{\widetilde{\theta}}

\addtocounter{section}{0}
\colorlet{myPineGreen}{PineGreen!150!white}

\title[A parametrix construction for the analysis of time-fractional PDEs]
{A parametrix construction\\for time-fractional partial differential equations}
\author{Sandro Coriasco}
\address{Dipartimento di Matematica ``G. Peano'', Universit\`a degli studi di Torino, Torino, Italy}
\email{sandro.coriasco@unito.it}
\author{Giovanni Girardi}
\address{Dipartimento di Ingegneria e Scienze, Università Telematica Universitas Mercatorum, Piazza Mattei, 10 - 00186 Roma, Italy}
\email{giovanni.girardi@unimercatorum.it}
	\author{Stevan Pilipovi\'c}
	\address{Department of Mathematics and Informatics, Faculty of Sciences,
		University of Novi Sad, Trg D. Obradovi\'ca 4, 
		RS-21000 Novi Sad, Serbia}
	\email{pilipovic@dmi.uns.ac.rs}

\begin{document}
\begin{abstract}
	We prove in detail how to construct the parametrix of a parameter-dependent family of pseudodifferential operators,
	appearing in the analysis of time-fractional partial differential equations. In particular, we perform a precise study of the dependence 
	from the parameter of the corresponding asymptotic expansion terms, as well as of the smoothing remainders. Moreover, we provide some
	results about the Laplace transform of vector-valued distributions, also appearing in the analysis of time-fractional partial differential
	equations. 

	\vspace{1cm}
	
	\begin{center}
		\textit{Dedicated to the memory of Luisa Zanghirati}
	\end{center}
\end{abstract}

\maketitle 

\section{Introduction}\label{sec:intro}

Fractional diffusion equations have become a standard tool for describing anomalous diffusion phenomena arising in physics, engineering and finance. Their mathematical analysis has attracted considerable attention during the last decades, both in the constant-coefficient setting and for more general classes of operators with variable coefficients.

In the constant-coefficient case, explicit solution formulae can be obtained by combining the Fourier and Laplace transforms, leading to representations in terms of Mittag--Leffler functions (see, e.g., \cite{GIL2000,GLU2000,Kem2021,MainardiBook}). These representations have been successfully applied to the study of qualitative properties of solutions, nonlinear perturbations and asymptotic behaviour (see, e.g., \cite{DAEP2019,DAG2022,DPVV2021}).

Similar questions have recently been investigated also for multi-term time-fractional equations, including models with fractional damping, see, for example, \cite{DAG2022,DAG2023}.

The situation changes substantially when variable coefficients are allowed. Explicit Fourier representations are no longer available, and the construction of solution operators requires a suitable pseudodifferential calculus. In the SG-framework this becomes particularly natural, since polynomially growing coefficients can be handled while preserving global mapping properties on weighted Sobolev spaces.

The main result of the paper is the construction of a parametrix for the SG-pseudodifferential family
\[
	s^r+\Op(a)=s^r+a(\cdot,D),
\]
where $r\in(0,1)$, $a=a(x,\xi)$ is a SG-hypoelliptic symbol and $s$ belongs to a suitable right half-plane of the complex plane. More precisely, we prove that the parametrix admits an explicit asymptotic expansion, whose coefficients are independent of the parameter $s$, while the corresponding remainder is smoothing and rapidly decaying with respect to $s$. The construction is uniform in $s$ and provides quantitative estimates for all symbol seminorms.

This symbolic description constitutes the basic ingredient for the analysis of fractional evolution equations associated with operators of the form
\[
	L=\partial_t^r+\Op(a)=\partial_t^r+a(\cdot,D), \quad r\in(0,1).
\]
Indeed, after applying the Laplace transform in time, the equation reduces to a ``suitable inversion'' of the operator family $s^r
+a(\cdot,D)=\Op(b_s)$, $b_s(x,\xi)=s^r+a(x,\xi)$, 
and the above construction makes it possible to define the corresponding solution operator through the inverse Laplace transform. 
Besides the symbolic construction of the parametrix, we establish several results on Laplace transforms of vector-valued distributions and on their interaction with pseudodifferential operators. These results provide the functional framework underlying the symbolic approach developed in this paper and may be of independent interest.
We employed the results proved in this paper in the analysis performed in \cite{CGP1}.

The paper is organized as follows. Section 2 reviews the SG-calculus of pseudodifferential operators. Section 3 is devoted to the construction of the parametrix and the proof of the main result. Finally, Section 4 develops the functional framework for vector-valued Laplace transforms and establishes several auxiliary results that provide the analytical tools needed in the study of fractional evolution equations considered in \cite{CGP1}.

\section*{acknowledgements}
The first author has been partially supported by the Italian Ministry of the University and Research - MUR, within the framework of the Call relating to the scrolling of the final rankings of the PRIN 2022 - Project Code 2022HCLAZ8, CUP D53C24003370006 (PI A. Palmieri, Local unit Sc. Resp. S. Coriasco). The first author also expresses
gratitude for the hospitality extended to him during his visit to the Department of Mathematics and Informatics, University of Novi Sad, Serbia, during A.Y. 2024/2025,
where part of this work was developed. The second author has been partially supported by INdAM GNAMPA Project, Grant Code CUP E55F22000270001.
The third author has been supported by the Serbian Academy of Sciences and Arts, project F10.

\section{The calculus of SG-pseudodifferential operators}\label{subs:sgcalc}
\setcounter{equation}{0}
We here recall some basic definitions and facts about the SG-calculus of pseudodifferential %and Fourier integral 
operators, through
standard material appeared, e.g., in \cite{ACS19b,CPS1} and elsewhere (sometimes with slightly different notational choices).
A detailed description of the calculus can be found in \cite{Cord}.
%We often employ the so-called \textit{japanese bracket} of $y\in\R^d$, given by $\langle y \rangle =\sqrt{1+|y|^2}$.

Denoting, as usual, $\norm{y}=\sqrt{1+|y|^2}$, $y\in\R^d$, in the estimates below, the class $S ^{m,\mu}=S ^{m,\mu}(\R^{d})=S ^{m,\mu}(\R^{d}\times\R^d)$ of SG-symbols 
of order $(m,\mu) \in \R^2$ is given by all the functions $a(x,\xi) \in C^\infty(\R^d\times\R^d)$
with the property
that, for any multiindices $\alpha,\beta \in \N_0^d$, there exist
constants $C_{\alpha\beta}>0$ such that the conditions 
\begin{equation}
	\label{eq:disSG}
	|D_x^{\alpha} D_\xi^{\beta} a(x, \xi)| \leq C_{\alpha\beta} 
	\x^{m-|\alpha|}\csi^{\mu-|\beta|},
	\qquad (x, \xi) \in \R^d \times \R^d,
\end{equation}
hold (cf. \cite{Cord,ME,PA72}). We often omit the base spaces $\R^d$, $\R^{d}\times\R^d$, etc., from the notation.

For $m,\mu\in\R$, $\ell\in\N_0$,
\begin{equation}\label{eq:SGseminorm}
	\vvvert a \vvvert^{m,\mu}_\ell
	= 
	\max_{|\alpha+\beta|\le \ell}\sup_{x,\xi\in\R^d}\x^{-m+|\alpha|} 
	                                                                     \csi^{-\mu+|\beta|}
	                                                                    | \partial^\alpha_x\partial^\beta_\xi a(x,\xi)|, \quad a\in\ S^{m,\mu},
\end{equation}
is a family of seminorms, defining  the Fr\'echet topology of $S^{m,\mu}$.

The corresponding
classes of pseudodifferential operators $\Op (S ^{m,\mu})=\Op (S ^{m,\mu}(\R^d))$ are given by
\begin{equation}\label{eq:psidos}
	(\Op(a)u)(x)=(a(.,D)u)(x)=(2\pi)^{-d}\int e^{\ii x\xi}a(x,\xi)\hat{u}(\xi)d\xi, \quad a\in S^{m,\mu}(\R^d),u\in\scS(\R^d),
\end{equation}
extended by duality to $\scS^\prime(\R^d)$.
The operators in \eqref{eq:psidos} form a
graded algebra with respect to composition, that is,
$$
\Op (S ^{m_1,\mu _1})\circ \Op (S ^{m_2,\mu _2})
\subseteq \Op (S ^{m_1+m_2,\mu _1+\mu _2}).
$$
The symbol $c\in S ^{m_1+m_2,\mu _1+\mu _2}$ of the composed operator $\Op(a)\circ\Op(b)$,
$a\in S ^{m_1,\mu _1}$, $b\in S ^{m_2,\mu _2}$, admits the asymptotic expansion
\begin{equation}
	\label{eq:comp}
	c(x,\xi)\sim \sum_{\alpha}\frac{i^{|\alpha|}}{\alpha!}\,D^\alpha_\xi a(x,\xi)\, D^\alpha_x b(x,\xi),
\end{equation}
which implies that the symbol $c$ equals $a\cdot b$ modulo $S ^{m_1+m_2-1,\mu _1+\mu _2-1}$.

Note that
\[
	 S ^{-\infty,-\infty}=S ^{-\infty,-\infty}(\R^{d})= \bigcap_{(m,\mu) \in \R^2} S ^{m,\mu} (\R^{d})
	 =\scS(\R^{2d}).
\]
For any $a\in S^{m,\mu}$, $(m,\mu)\in\R^2$,
$\Op(a)$ is a linear continuous operator from $\scS(\R^d)$ to itself, extending to a linear
continuous operator from $\scS^\prime(\R^d)$ to itself, and from
$H^{z,\zeta}(\R^d)$ to $H^{z-m,\zeta-\mu}(\R^d)$,
where $H^{z,\zeta}=H^{z,\zeta}(\R^d)$,
$(z,\zeta) \in \R^2$, denotes the  Sobolev-Kato (or \textit{weighted Sobolev}) space
\begin{equation}\label{eq:skspace}
  	H^{z,\zeta}(\R^d)= \{u \in \scS^\prime(\R^{n}) \colon \|u\|_{z,\zeta}=
	\|{\japdot}^z\pdd^\zeta u\|_{L^2}< \infty\},
\end{equation}
(here $\pdd^\zeta$ is understood as a pseudodifferential operator) with the naturally induced Hilbert norm. When $z\ge z^\prime$ and $\zeta\ge\zeta^\prime$, the continuous embedding 
$H^{z,\zeta}\hookrightarrow H^{z^\prime,\zeta^\prime}$ holds true. It is compact when $z>z^\prime$ and $\zeta>\zeta^\prime$.
Since $H^{z,\zeta}=\japdot^z\,H^{0,\zeta}=\japdot^z\, H^\zeta$, with $H^\zeta$ the usual Sobolev space of order $\zeta\in\R$, we 
find $\zeta>k+\dfrac{d}{2} \Rightarrow H^{z,\zeta}\hookrightarrow C^k(\R^d)$, $k\in\N_0$. One actually finds
\begin{equation}\label{eq:spdecomp}
	\bigcap_{z,\zeta\in\R}H^{z,\zeta}(\R^d)=H^{\infty,\infty}(\R^d)=\scS(\R^d),
	\quad
	\bigcup_{z,\zeta\in\R}H^{z,\zeta}(\R^d)=H^{-\infty,-\infty}(\R^d)=\scS^\prime(\R^d),
\end{equation}
as well as, for the space of \textit{rapidly decreasing distributions}, see \cite[Chap. VII, \S 5]{schwartz}, 
\begin{equation}\label{eq:rdd}
	\scS^\prime(\R^d)_\infty=\bigcap_{z\in\R}\bigcup_{\zeta\in\R}H^{z,\zeta}(\R^d).
\end{equation}
The continuity property of
the elements of $\Op(S^{m,\mu})$ on the scale of spaces $H^{z,\zeta}(\R^d)$, $(m,\mu),(z,\zeta)\in\R^2$, is expressed 
more precisely in the next theorem.
\begin{thm}[{\cite[Chap. 3, Theorem 1.1]{Cord}}] \label{thm:sobcont}
	Let $a\in S^{m,\mu}(\R^d)$, $(m,\mu)\in\R^2$. Then, for any $(z,\zeta)\in\R^2$, 
	$\Op(a)\in\mathcal{L}(H^{z,\zeta}(\R^d),H^{z-m,\zeta-\mu}(\R^d))$, and there exists a constant $C>0$,
	depending only on $d,m,\mu,z,\zeta$, such that
	\begin{equation}\label{eq:normsob}
		\|\Op(a)\|_{\scrL(H^{z,\zeta}(\R^d), H^{z-m,\zeta-\mu}(\R^d))}\le 
		C\vvvert a \vvvert_{\left[\frac{d}{2}\right]+1}^{m,\mu},
	\end{equation}
	where $[t]$ denotes the integer part of $t\in\R$.
\end{thm}
The class $\caO(m,\mu)$ of the \textit{operators of order $(m,\mu)$} is introduced as follows (see, e.g., \cite[Chap. 3, \S 3]{Cord}).
\begin{defn}\label{def:ordmmuopr}
	A linear continuous operator $A\colon\scS(\R^d)\to\scS(\R^d)$
	belongs to the class $\caO(m,\mu)$, $(m,\mu)\in\R^2$, of the operators of order $(m,\mu)$ if, for any $(z,\zeta)\in\R^2$,
	it extends to a linear continuous operator $A_{z,\zeta}\colon H^{z,\zeta}(\R^d)\to H^{z-m,\zeta-\mu}(\R^d)$. We also define
	\[
		\caO(\infty,\infty)=\bigcup_{(m,\mu)\in\R^2} \caO(m,\mu), \quad
		\caO(-\infty,-\infty)=\bigcap_{(m,\mu)\in\R^2} \caO(m,\mu).		
	\]
\end{defn}
\begin{rem}\label{rem:O}
	\begin{enumerate}
		\item[(i)] Trivially, any $A\in\caO(m,\mu)$ admits a linear continuous extension 
		$A_{\infty,\infty}\colon\scS^\prime(\R^d)\to\scS^\prime(\R^d)$. In fact, in view of \eqref{eq:spdecomp}, it is enough to set
		$A_{\infty,\infty}|_{H^{z,\zeta}(\R^d)}= A_{z,\zeta}$.
		\item[(ii)] Theorem \ref{thm:sobcont} implies $\Op(S^{m,\mu}(\R^d))\subset\caO(m,\mu)$, $(m,\mu)\in\R^2$.
		\item[(iii)] $\caO(\infty,\infty)$ and $\caO(0,0)$ are algebras under operator multiplication, $\caO(-\infty,-\infty)$ is an ideal
		of both  $\caO(\infty,\infty)$ and $\caO(0,0)$, and 
		$\caO(m_1,\mu_1)\circ\caO(m_2,\mu_2)\subset\caO(m_1+m_2,\mu_1+\mu_2)$.
	\end{enumerate}
\end{rem}
\noindent
The following characterization of the class $\caO(-\infty,-\infty)$ is often useful.
\begin{prop}[{\cite[Ch. 3, Prop. 3.4]{Cord}}] \label{thm:smoothing}
	The class $\caO(-\infty,-\infty)$ coincides with $\Op(S^{-\infty,-\infty}(\R^d))$ and with the class of smoothing operators,
	that is, the set of all the linear continuous operators $A\colon\scS^\prime(\R^d)\to\scS(\R^d)$. All of them coincide with the
	class of linear continuous operators $A$ admitting a Schwartz kernel $k_A$ belonging to $\scS(\R^{2d})$. 
\end{prop}
An operator $A=\Op(a)$ and its symbol $a\in S ^{m,\mu}$ are called \emph{elliptic}
(or $S ^{m,\mu}$-\emph{elliptic} or \emph{md-elliptic}) if there exists $R\ge0$ such that
%and
%
\[
	C\x^{m} \csi^{\mu}\le |a(x,\xi)|,\qquad 
	|x|+|\xi|\ge R,
\] 
for some constant $C>0$. If $R=0$, $a^{-1}$ is everywhere well-defined and smooth, and $a^{-1}\in S ^{-m,-\mu}$.
If $R>0$, then $a^{-1}$ can be extended to the whole of $\R^{2d}$ so that the extension $\widetilde{a}_{-1}$ satisfies $\widetilde{a}_{-1}\in S ^{-m,-\mu}$.
An elliptic SG-operator $A \in \Op (S ^{m,\mu})$ admits a
parametrix $A_{-1}\in \Op (S ^{-m,-\mu})$ such that
\[
A_{-1}A=I + R_1, \quad AA_{-1}= I+ R_2,
\]
for suitable $R_1, R_2\in\Op(S^{-\infty,-\infty})$, where $I$ denotes the identity operator. 
In such a case, $A$ turns out to be a Fredholm
operator on the scale of functional spaces $H^{z,\zeta}$,
$(z,\zeta)\in\R^2$.

In a similar fashion, an operator $A=\Op(a)$ and its symbol $a\in S ^{m,\mu}$ are called \emph{SG-hypoelliptic}
(or \textit{$(m,\mu,m^\prime,\mu^\prime)$( SG-)hypoelliptic}) if there exists $R\ge0$, $m^\prime,\mu^\prime\in\R$, $m^\prime\le m$, $\mu^\prime\le \mu$, such that
%and
%
\begin{equation}\label{eq:hypoell_a}
	C\x^{m^\prime} \csi^{\mu^\prime}\le \Re(a(x,\xi)),\qquad 
	|x|+|\xi|\ge R,
\end{equation}
for some constant $C>0$ and, for all multi-indices $\alpha, \beta\in \N^d$ there exist constants $C_{\alpha \beta}>0$ such that 
\begin{equation}\label{eq:hypoell_b}
\left| \frac{\partial_x^\alpha\partial_\xi^\beta a(x,\xi)}{\Re(a(x,\xi))}\right|\leq C_{\alpha \beta} \<x\>^{-|\alpha|}\<\xi\>^{-|\beta|},
\end{equation}
for any $(x,\xi)\in \R^{d}\times \R^d$ with $|x|+|\xi|\geq R$. Notice that for any hypoelliptic symbol $a$
there exists $a_0\in C_0^\infty(\R^d\times\R^d)$ such that $\widetilde{a}=a+a_0$ satisfies \eqref{eq:hypoell_a}
with $R=0$ and a different constant $C>0$. That is, the lower bound for the symbol $\widetilde{a}$ holds true on the whole $\R^d\times\R^d$. 
Indeed, assume, without loss of generality, that $\Re(a(x,\xi))>0$ for $|x|+|\xi|\ge R$, set $Q_R=\{(x,\xi)\in\R^d\times\R^d\colon |x|+|\xi|\le R\}$, 
and let $\chi\in C_0^\infty(\R^d\times\R^d)$ be a cut-off function such that $0\le\chi\le1$,
$\supp\chi\subseteq Q_{2R}$, and $\chi|_{Q_\frac{3R}{2}}\equiv1$. Let also 
\[
	\delta=\min_{|x|+|\xi|\le R} \<x\>^{-m^\prime}\<\xi\>^{-\mu^\prime}\Re(a(x,\xi))\in\R\text{ so that } 
	\forall x,\xi\in\R^d \; |x|+|\xi|\le R\Rightarrow \Re(a(x,\xi))\ge\delta\x^{m^\prime}\csi^{\mu^\prime},
\]
and set $a_0(x,\xi)=(1+|\delta|)\chi(x,\xi)\x^{m^\prime}\csi^{\mu^\prime}\in S^{-\infty,-\infty}\Leftarrow a_0\in C_0^\infty$.
Obviously, by construction, $a_0\ge0$, so \eqref{eq:hypoell_a} holds true for $\widetilde{a}$ when $|x|+|\xi|\ge R$, with the same constant $C$.
Moreover, for any $|x|+|\xi|\le R$,
\[
	\Re(\widetilde{a}(x,\xi))=\Re(a(x,\xi))+a_0(x,\xi)\ge \delta\x^{m^\prime}\csi^{\mu^\prime}+(1+|\delta|)\x^{m^\prime}\csi^{\mu^\prime}\ge\x^{m^\prime}\csi^{\mu^\prime}.
\]
The desired estimate for $\widetilde{a}$ then follows choosing $\widetilde{C}=\min\{1,C\}$. Also \eqref{eq:hypoell_b} holds true for 
$\widetilde{a}$, with different constants, for any $x,\xi\in\R^d$. Finally, also $(m,\mu,m^\prime,\mu^\prime)$-hypoelliptic operators admit a parametrix
(see, e.g., \cite[Ch. 2, \S\ 2]{Cord}).

\section{Construction of the parametrix}\label{app:cs_parametrix}
\setcounter{equation}{0}

We consider a real-valued symbol $a\in S^{m,\mu}(\R^d\times\R^d)$ which 
%\textcolor{blue}{non-negative}, 
satisfies the following assumptions:

\begin{enumerate}[label=\textbf{(H\arabic*)},ref=\textbf{(H\arabic*)}]
\item \label{hyp:a_symbol} there exist $m,\mu\in(0,+\infty)$ such that $a\in S^{m,\mu}(\R^d\times \R^d)$;
\item \label{hyp:a_hypoellyptic} $a$ is non-negative, and
there exist $R>0$, $m^\prime\in [0,m]$ and  $\mu^\prime\in [0,\mu]$ such that, for any $(x,\xi)\in \R^{d}\times \R^d$ with $|x|+|\xi|\geq R$,
\begin{equation}
a(x,\xi)\geq C \<x\>^{m^\prime}\<\xi\>^{\mu^\prime},
\end{equation}
for some contant $C>0$ independent of $x$ and $\xi$;
\item \label{hyp:a_derivatives} for all multi-indices $\alpha, \beta\in \N^d$ there exist constants $C_{\alpha \beta}>0$ such that 
\begin{equation}
 \frac{\left|\partial_x^\alpha\partial_\xi^\beta a(x,\xi)\right|}{a(x,\xi)}\leq C_{\alpha \beta} \<x\>^{-|\alpha|}\<\xi\>^{-|\beta|},
\end{equation}
for any $(x,\xi)\in \R^{d}\times \R^d$ with $|x|+|\xi|\geq R$.
\end{enumerate}
In short, $a$ is non-negative and $(m,\mu,m^\prime,\mu^\prime)$-hypoelliptic in the SG-calculus, with $m,\mu>0$, $m^\prime\in[0,m]$, $\mu^\prime\in[0,\mu]$.

\begin{rem}
\label{rem:bs_in_SG}
Let $s\in \C$ be such that $\Re s>\lambda>0$ and define $b_s(x,\xi):=s^r+a(x,\xi)$, $x,\xi\in\R^d$.
If $\lambda$ is sufficiently large then $b_s\not=0$, in view of assumption \ref{hyp:a_hypoellyptic} 
(see Remark \ref{rem:bs_hypoelliptic} below and the last paragraph of Section \ref{subs:sgcalc}),
and belongs to $S^{m,\mu}(\R^d\times \R^d)$, where $m$ and $\mu$ are given in assumption \ref{hyp:a_symbol}. Indeed, for every $\alpha, \beta \in \N^d$ there exists $\tilde{C}_{\alpha \beta}>0$ such that the inequality
\[|\partial_x^\alpha \partial_\xi^\beta a(x,\xi)|\leq \tilde{C}_{\alpha \beta}\<x\>^{m-|\alpha|}\<\xi\>^{\mu-|\beta|}\]
holds true for any $(x,\xi)\in \R^d\times \R^d$, as a consequence of assumption \ref{hyp:a_symbol}. Then, we may estimate
\[ |b_s(x,\xi)|\leq |s|^r (1+|s|^{-r}a(x,\xi)) \leq |s|^r \Big(1+\frac{\tilde C_{0 0} \<x\>^{m}\<\xi\>^\mu}{\lambda^r }\Big)\leq 2|s|^r \<x\>^m\<\xi\>^\mu,\]
provided that $\lambda\ge \tilde C_{00}^\frac{1}{r}$. Moreover, for any $\alpha, \beta \in \N^d$ with $|\alpha|+|\beta|\geq 1$, it holds 
\[ |\partial_x^\alpha\partial_\xi^\beta b_s(x,\xi)|=|\partial_x^\alpha\partial_\xi^\beta a(x,\xi)|\leq \tilde{C}_{\alpha \beta}\<x\>^{m-|\alpha|}\<\xi\>^{\mu-|\beta|}.\]
\end{rem}
\begin{rem}
\label{rem:bs_hypoelliptic}
For every $(x,\xi)\in \R^d\times \R^d$ with $|x|+|\xi|\geq R$ and $s=|s| e^{i \theta}$, $\Re s>\lambda>0$, it holds
\begin{equation}
	\label{eq:bs_hypoelliptic}
	\Re(b_s(x,\xi))= |s|^r\cos( r \theta ) + a(x,\xi)> a(x,\xi)\ge C \<x\>^{m^\prime}\<\xi\>^{\mu^\prime}\ge C>0,
\end{equation}
as a consequence of assumption \ref{hyp:a_hypoellyptic}. Indeed, being $\Re s>\lambda>0$ and $r\in (0,1)$, we know that $r\theta\in (-r\pi/2, r\pi/2) $, 
so that $\cos(r\theta)>\cos(r\pi/2)> 0$. Actually, for $\Re s>\lambda>0$, \eqref{eq:bs_hypoelliptic} holds true for
arbitrary $x,\xi\in\R^d$, possibly reducing $C$ to a smaller $\widetilde{C}>0$. 
In fact, $\Re s>\lambda>0$ implies $\Re s^r =|s|^r\cos(r\vartheta)> \lambda^r\cos(r\pi/2)=\kappa>0$. By assumption \ref{hyp:a_hypoellyptic},
$a$ is non-negative, and $|x|+|\xi|\le R$ implies $1\le\x^{m^\prime}\csi^{\mu^\prime}\le\langle R\rangle^{m^\prime+\mu^\prime}$, so that, 
\begin{align*}
	|x|+|\xi|\le R \Rightarrow \Re(b_s(x,\xi))& = |s|^r\cos( r \theta ) + a(x,\xi) \\
& \ge |s|^r\cos( r \theta )
	>\kappa\ge\frac{\kappa}{\langle R\rangle^{m^\prime+\mu^\prime}} \<x\>^{m^\prime}\<\xi\>^{\mu^\prime}.
\end{align*}
To achieve the lower bound \eqref{eq:bs_hypoelliptic} for arbitrary $x,\xi\in\R^d$
it is then enough to substitute, in the right hand side, the constant $C$ with the constant $\widetilde{C}=\min\left\{C,\dfrac{\kappa}{\langle R\rangle^{m^\prime+\mu^\prime}}\right\}>0$.

Moreover, by assumption \ref{hyp:a_derivatives},
for any $\alpha, \beta\in \N^d$ multi-indices, $|\alpha|+|\beta|\ge1$, being $|b_s(x,\xi)|\geq |\Re(b_s(x,\xi))|$, we may estimate   
\begin{equation}\label{eq:bsders}
\left| \frac{\partial_x^\alpha\partial_\xi^\beta b_s(x,\xi)}{b_s(x,\xi)}\right| 
\le \left|\frac{\partial_x^\alpha\partial_\xi^\beta a(x,\xi)}{\Re(b_s(x,\xi))}\right|
%\le \left|\frac{\partial_x^\alpha\partial_\xi^\beta a(x,\xi)}{a(x,\xi)}\right| 
\leq \widetilde{C}_{\alpha \beta} \<x\>^{-|\alpha|}\<\xi\>^{-|\beta|},
\end{equation}
for any $(x,\xi)\in\R^d\times\R^d$, with suitable $\widetilde{C}_{\alpha \beta}>0$.  
Indeed, for $|x|+|\xi|\ge R$, since $\Re s>\lambda >0$ implies $|\Re(b_s(x,\xi))|> a(x,\xi)$, \eqref{eq:bsders} holds true
choosing $\widetilde{C}_{\alpha \beta}=C_{\alpha\beta}>0$, the constants given in assumption \ref{hyp:a_derivatives}. 
For $|x|+|\xi|\le R$, since $\kappa=\dfrac{\kappa}{\langle R\rangle^{m+\mu}}\langle R\rangle^{m+\mu}\ge\dfrac{\kappa}{\langle R\rangle^{m+\mu}}\x^m\csi^\mu$
and $|\Re(b_s(x,\xi))|> \kappa$, recalling assumption \ref{hyp:a_symbol}, we find
\[
	\left| \frac{\partial_x^\alpha\partial_\xi^\beta b_s(x,\xi)}{b_s(x,\xi)}\right|
	\le \left|\frac{\partial_x^\alpha\partial_\xi^\beta a(x,\xi)}{\Re(b_s(x,\xi))}\right|
	\le \frac{\x^{m-|\alpha|}\csi^{\mu-|\beta|}}{\dfrac{\kappa}{\langle R\rangle^{m+\mu}}\x^m\csi^\mu}
	=\frac{\langle R\rangle^{m+\mu}}{\kappa} \x^{-|\alpha|}\csi^{-|\beta|}.
\]
We conclude that to achieve \eqref{eq:bsders} for arbitrary $x,\xi\in\R^d$ it is enough to choose 
$\widetilde{C}_{\alpha\beta}=\max\left\{C_{\alpha\beta},\dfrac{\langle R\rangle^{m+\mu}}{\kappa}\right\}$.
\end{rem}
\begin{rem}\label{rem:bs_hypoell}
By Remarks \ref{rem:bs_in_SG} and \ref{rem:bs_hypoelliptic}, we derive that, under the assumptions 
\ref{hyp:a_symbol}, \ref{hyp:a_hypoellyptic}, and \ref{hyp:a_derivatives}, the symbol $b_s(x,\xi)=s^r + a(x,\xi)$, $x,\xi\in\R^d$, 
is $(m,\mu,m^\prime,\mu^\prime)$ SG-hypoelliptic (see the last paragraph of Section \ref{subs:sgcalc}), 
for all $s\in \C$ with $\Re s>\lambda>0$, $\lambda$ sufficiently large. From now on we
will always assume $\Re s>\lambda>0$ so that this property holds true.
As a consequence, $\Op(b_s)$ admits a parametrix $\Op(c_s)$. 
\end{rem}

We now construct the parametrix $\Op(c_s)$ of $\Op(b_s)$, refining the 
classical approach (cf., e.g., \cite[Theorem 2.5]{Cord}). This is a variant of the classical results for the construction of the parametrix to the resolvent of
suitable elliptic operators, originally due to Seeley (see \cite{Seeley67}; see also \cite{MSS06} for the case of elliptic SG-classical operators). 

\begin{prop}\label{prop:cs_parametrix}
Let $a\in S^{m,\mu}(\R^d\times\R^d)$ be real-valued and satisfy assumptions 
	\ref{hyp:a_symbol}, \ref{hyp:a_hypoellyptic} and \ref{hyp:a_derivatives}. Setting $b_s(x,\xi)=s^r+a(x,\xi)$, $r\in(0,1)$,
there exists a family of symbols $c_s\in S^{-m^\prime,-\mu^\prime}(\R^d\times \R^d)$ such that, for any $s\in \C_\lambda=\{s\in\C\colon \Re s>\lambda>0\}$,
$\lambda$ sufficiently large, $\Op(c_s)$ is a parametrix of $\Op(b_s)$, that is,
\begin{equation}\label{eq:cs_par}
	\Op(b_s)\Op(c_s)=I+\Op(r_{1s}), \quad \Op(c_s)\Op(b_s)=I+\Op(r_{2s}),
\end{equation}
for suitable $r_{1s}, r_{2s}\in S^{-\infty, -\infty}(\R^d\times \R^d)$. More precisely, there exist symbols 
$Q_j\in S^{(j+2)m-j-1, (j+2)\mu-j-1}$, $j\in \N$, independent of $s$, such that $c_s$ is given by the asymptotic sum
\begin{equation}
\label{eq:cs_asymptoticsum}
c_s(x,\xi)\sim \frac{1}{s^r+a(x,\xi)} + \sum_{j\in\N} \frac{Q_j(x,\xi)}{[s^r+a(x,\xi)]^{j+3}} = \sum_{j\in\N}\frac{A_j(x,\xi)}{[s^r+a(x,\xi)]^{j+1}},
\end{equation}
$A_0\equiv 1$, $A_1\equiv0$, $A_j=Q_{j-2} \in S^{jm-j+1,j\mu-j+1}$, $j\ge2$, and satisfies, for any $k\in \N$, the estimates
\[ 
	||| c_s|||^{-m^\prime,-\mu^\prime}_k\leq C_k,
\]
for suitable constants $C_k>0$, independent of $s\in\C_\lambda$, and the seminorms defined in \eqref{eq:SGseminorm}. 
In particular, for any $j\in \N$, $j\ge2$, the symbol $A_j$ admits an asymptotic expansion of the form
\begin{equation}
\label{eq:asymptotic_Aj} 
A_j\sim \sum_{|\theta|>j-2}\tilde{P}_{j}^{\theta\theta}, \quad \text{where} \quad \tilde{P}_{j}^{\theta\theta}\in \vspan\left\{\partial^{\theta_1}_x\partial^{\sigma_1}_\xi a(x,\xi)\cdots\partial^{\theta_j}_x\partial^{\sigma_j}_\xi a(x,\xi)
		\colon \sum_{k=1}^j\theta_k=\sum_{k=1}^j\sigma_k=\theta \right\}.
\end{equation}
\end{prop}
%
%\begin{rem}
%In Proposition \ref{prop:cs_parametrix} for any $\ell, \rho\in \R$ and $p\in S^{\ell, \rho}(\R^d\times \R^d)$ we are considering the family of seminorms
%%
%\[ |||p|||_k=\sup_{|\alpha|+|\beta|\leq k}\sup_{(x,y)\in \R^{2d}}|\partial_x^\alpha \partial_\xi^\beta p(x,\xi)|\<x\>^{-\ell+|\alpha|}\<\xi\>^{-\rho+|\beta|},\]
%%
%with $k\in \N$, which defines a Fr\'echet topology on $S^{\ell,\rho}(\R^d\times \R^d)$ (notice that these seminorms are actually norms).
%\end{rem}
%%

We need the next Lemma \ref{lem:deriv}, which can be proved by induction on the heights of the
involved multi-indices. The details are left for the reader.
\begin{lem}\label{lem:deriv}
	Let $p\in S^{m,\mu}(\R^d\times \R^d)$ and $\theta,\sigma\in\Z_+^d$ be such that $|\theta+\sigma|\ge1$.  Assume  that $p(x,\xi)$ is different from 
	zero for $x,\xi\in\R^d$. Then,
	\[
		\partial^\theta_x\partial^\sigma_\xi\left[\frac{1}{p(x,\xi)}\right]=\sum_{j=1}^{|\theta+\sigma|}\frac{P^{\theta\sigma}_j(x,\xi)}{[p(x,\xi)]^{j+1}},
	\]
	with
	\begin{equation}\label{eq:Pst}
		P^{\theta\sigma}_j(x,\xi)\in\vspan\left\{\partial^{\theta_1}_x\partial^{\sigma_1}_\xi p(x,\xi)\cdots\partial^{\theta_j}_x\partial^{\sigma_j}_\xi p(x,\xi)
		\colon \sum_{k=1}^j\theta_k=\theta,\sum_{k=1}^j\sigma_k=\sigma\right\},
	\end{equation}
	so that $P^{\theta\sigma}_j\in S^{jm-|\theta|,j\mu-|\sigma|}(\R^d\times \R^d)$. Moreover, 
	if $p$ is $(m,\mu,m^\prime,\mu^\prime)$ SG-hypoelliptic, it also holds
	\[
		\frac{P^{\theta\sigma}_j(x,\xi)}{[p(x,\xi)]^j}\in S^{-|\theta|,-|\sigma|}(\R^d\times \R^d), 
		\frac{P^{\theta\sigma}_j(x,\xi)}{[p(x,\xi)]^{j+1}}\in S^{-m^\prime-|\theta|,-\mu^{\prime}-|\sigma|}(\R^d\times \R^d), 
		\quad \theta,\sigma\in\Z_+^d,j=1,\dots,|\theta+\sigma|.
	\]
	More generally, for any $\tau\in\N$ and $\theta,\sigma\in\Z_+^d$ such that $|\theta+\sigma|\ge1$,
	\[
		\partial^\theta_x\partial^\sigma_\xi\left\{\frac{1}{[p(x,\xi)]^\tau}\right\}=\sum_{j=1}^{|\theta+\sigma|}\frac{P^{\tau,\theta\sigma}_{j}(x,\xi)}{[p(x,\xi)]^{\tau+j}},
	\]
	with
	\[
		P^{\tau,\theta\sigma}_{j}(x,\xi)\in\vspan\left\{\partial^{\theta_1}_x\partial^{\sigma_1}_\xi p(x,\xi)\cdots\partial^{\theta_j}_x\partial^{\sigma_j}_\xi p(x,\xi)
		\colon \sum_{k=1}^j\theta_k=\theta,\sum_{k=1}^j\sigma_k=\sigma\right\}.
	\]
	Moreover, if $p$ is $(m,\mu,m^\prime,\mu^\prime)$ SG-hypoelliptic on $\R^d\times\R^d$, it also holds
	\[
		\frac{P^{\tau,\theta\sigma}_{j}(x,\xi)}{[p(x,\xi)]^{\tau+j}}\in S^{-\tau m^\prime-|\theta|,-\tau\mu^{\prime}-|\sigma|}(\R^d\times \R^d), 
		\quad \theta,\sigma\in\Z_+^d,j=1,\dots,|\theta+\sigma|.
	\]
\end{lem}
\noindent
\begin{proof}[Proof of Proposition \ref{prop:cs_parametrix}]
We refine the usual parametrix construction, making more explicit the structure of the employed asymptotic expansions in terms
of the powers of $b_s(x,\xi)=s^r+a(x,\xi)$, $s\in\C_\lambda$. We split the proof into various steps.

\noindent
i) Set
\[ 
	c_{0s}(x,\xi)=\frac{1}{b_s(x,\xi)}.
\]
From Remark \ref{rem:bs_hypoelliptic}, we immediately see that there exists $C>0$ independent of $s$ such that
\[ |c_{0s}(x,\xi)|\leq C^{-1}\<x\>^{-m^\prime}\<\xi\>^{-\mu^\prime},\]
for every $(x,\xi)\in \R^d\times \R^d$. 
By Lemma \ref{lem:deriv}, SG estimates hold true for all the derivatives
of $c_{0s}$, $s\in\C_\lambda$.
\noindent
ii) By the calculus, it holds
\[ 
	\Op(b_s)\circ \Op(c_{0s})=I+\Op(e_{0s}),
\]
modulo $S^{-\infty,-\infty}$, where, for $s\in\C_\lambda$, $e_{0s}\in S^{-1,-1}$ and, in view of Lemma
\ref{lem:deriv}, is given by 
\begin{equation}\label{eq:e0sae1st}
\begin{aligned}
	e_{0s}(x,\xi)&\sim \sum_{|\theta|>0} \frac{i^{|\theta|}}{\theta!} D_\xi^\theta b_s(x,\xi)D^\theta_x c_{0s}(x,\xi) 
	=\sum_{|\theta|>0} \frac{i^{|\theta|}}{\theta!} D_\xi^\theta a(x,\xi)\sum_{j=1}^{|\theta|}\frac{P^{\theta0}_j(x,\xi)}{[b_s(x,\xi)]^{j+1}}
	\\
	&=\sum_{|\theta|>0}\sum_{j=1}^{|\theta|}  \frac{\dfrac{i^{|\theta|}}{\theta!} \, D_\xi^\theta a(x,\xi)\, P^{\theta0}_j(x,\xi)}{[b_s(x,\xi)]^{j+1}}
	=\sum_{|\theta|>0}\sum_{j=1}^{|\theta|}  \frac{\wtP^{\theta\theta}_{j+1,0}(x,\xi)}{[b_s(x,\xi)]^{j+1}}.
\end{aligned}
\end{equation}
where the symbols $\wtP^{\theta\theta}_{j+1,0}\in S^{(j+1)m-|\theta|,(j+1)\mu-|\theta|}$, $j\ge1$, have the form \eqref{eq:Pst}. We can rewrite, equivalently, 
\begin{equation}\label{eq:e0sae2nd}
	e_{0s}(x,\xi)\sim \frac{1}{[b_s(x,\xi)]^2}\sum_{|\theta|>0}\wtP^{\theta\theta}_{2,0}(x,\xi) 
	+ \sum_{j\ge1}\frac{1}{[b_s(x,\xi)]^{j+2}}\sum_{|\theta|>j}\wtP^{\theta\theta}_{j+2,0}(x,\xi)=e_{0sp}(x,\xi)+e_{0sr}(x,\xi).
\end{equation}
Indeed, the terms in the asymptotic expansions \eqref{eq:e0sae1st} and \eqref{eq:e0sae2nd} are the same 
(taken in a different order), and, in view of Lemma \ref{lem:deriv}, 
the sums with respect to $|\theta|$ of the $\wtP^{\theta\theta}_{j+2,0}$, $j\in\N$, are themselves SG-asymptotic expansions, identifying 
symbols $Q_{0,j}\in S^{(j+2)m-j-1,(j+2)\mu-j-1}$, $j\in\N$, modulo $S^{-\infty,-\infty}$, independent from $s$. Again by 
Lemma \ref{lem:deriv}, it follows, for $s\in\C_\lambda$, $e_{0sp}\in S^{-1,-1}$ and $e_{0sr}\in S^{-2,-2}$, since 
\begin{equation}\label{eq:Qbsym}
	\frac{Q_{0,j}(x,\xi)}{[b_s(x,\xi)]^{j+2}}\in S^{-j-1,-j-1}, \quad j\in\N,
\end{equation}
so that the summation with respect to $j$ which defines $e_{0sr}$ is again a SG-asymptotic expansion. We sketch the proof of
\eqref{eq:Qbsym}. By the definition of $Q_{0,j}$, for any $j\ge0$, $N\ge 1$,
\[
	Q_{0,j}=\sum_{|\theta|=j+1}^{j+N} \wtP^{\theta\theta}_{j+2,0} + R^{N}_{j+2,0}, \quad
	R^N_{j+2,0}\in S^{(j+2)m-j-1-N,(j+2)\mu-j-1-N}.
\]
Then, choosing $N>\max\{1,(j+2)(m-m^\prime),(j+2)(\mu-\mu^\prime)\}$, for any $x,\xi\in\R^d$, recalling Remark \ref{rem:bs_hypoell},
\begin{align*}
	\left| \frac{Q_{0,j}(x,\xi)}{[b_s(x,\xi)]^{j+2}} \right| &\le \sum_{|\theta|=j+1}^{j+N} \left| \frac{\wtP^{\theta\theta}_{j+2,0}(x,\xi)}{[b_s(x,\xi)]^{j+2}} \right|
	 + \left| \frac{R^{N}_{j+2,0}(x,\xi)}{[b_s(x,\xi)]^{j+2}} \right|
	 \\
	 &\lesssim \sum_{|\theta|=j+1}^{j+N} \jap{x}^{-|\theta|}\jap{\xi}^{-|\theta|}+
	 \frac{\jap{x}^{(j+2)m-j-1-N}\jap{\xi}^{(j+2)\mu-j-1-N}}{\jap{x}^{(j+2)m^\prime}\jap{\xi}^{(j+2)\mu^\prime}}
	 \\
	 &\lesssim \jap{x}^{-j-1}\jap{\xi}^{-j-1}
\end{align*}
Similarly, for $\alpha,\beta\in\Z_+^d$ such that $|\alpha+\beta|=1$, choosing again 
$N>\max\{1,(j+2)(m-m^\prime),(j+2)(\mu-\mu^\prime)\}$, and recalling Remark \ref{rem:bs_hypoelliptic},
for any $x,\xi\in\R^d$, 
\begin{align*}
	\left| \partial^\alpha_x\partial^\beta_\xi\left\{\frac{Q_{0,j}(x,\xi)}{[b_s(x,\xi)]^{j+2}}\right\} \right| 
	&\le \sum_{|\theta|=j+1}^{j+N} \left| \partial^\alpha_x\partial^\beta_\xi\left\{\frac{\wtP^{\theta\theta}_{j+2,0}(x,\xi)}{[b_s(x,\xi)]^{j+2}} \right\}\right|
	 + \left| \partial^\alpha_x\partial^\beta_\xi\left\{\frac{R^{N}_{j+2,0}(x,\xi)}{[b_s(x,\xi)]^{j+2}} \right\}\right|
	 \\
	 &\lesssim \sum_{|\theta|=j+1}^{j+N} \jap{x}^{-|\theta|-|\alpha|}\jap{\xi}^{-|\theta|-|\beta|}+
	  \left| \frac{\partial^\alpha_x\partial^\beta_\xi R^{N}_{j+2,0}(x,\xi)}{[b_s(x,\xi)]^{j+2}} \right|+
	  \left| \frac{\partial^\alpha_x\partial^\beta_\xi a(x,\xi)}{b_s(x,\xi)} \cdot \frac{R^{N}_{j+2,0}(x,\xi)}{[b_s(x,\xi)]^{j+2}} \right|
	 \\
	 &\lesssim  \jap{x}^{-j-1-|\alpha|}\jap{\xi}^{-j-1-|\beta|}+
	 \jap{x}^{(j+2)(m-m^\prime)-j-1-|\alpha|-N}\jap{\xi}^{(j+2)(\mu-\mu^\prime)-j-1-|\beta|-N}
	 \\
	 &+\jap{x}^{-|\alpha|}\jap{\xi}^{-|\beta|}\jap{x}^{(j+2)(m-m^\prime)-j-1-N}\jap{\xi}^{(j+2)(\mu-\mu^\prime)-j-1-N}
	 \\
	 &\lesssim \jap{x}^{-j-1-|\alpha|}\jap{\xi}^{-j-1-|\beta|}.
\end{align*}
The estimates for general $\alpha,\beta\in\Z_+^d$ follow by an induction argument, 
again employing Lemma \ref{lem:deriv} and the hypoellipticity hypothesis.
\\

\noindent
iii) Set
\[
	c_{1s}(x,\xi)=-\frac{Q_{0,0}(x,\xi)}{[b_s(x,\xi)]^3}.
\]
By an argument completely similar to the one sketched at the end of the previous step, we see that, for $s\in\C_\lambda$,
$c_{1s}\in S^{-m^\prime-1,-\mu^\prime-1}$.
The same computations also show that the seminorm $||| c_{1s}|||^{-m^\prime-1,-\mu^\prime-1}_k$ is uniformly bounded with respect to $s\in \C_\lambda$, 
for any $k\in \N$. Moreover, by the calculus,
\[
	\Op(b_s)\circ \Op(c_{0s}+c_{1s})=I+\Op(e_{0sp})+\Op(e_{0sr})-\Op(e_{0sp})+\Op(\wte_{0sr})=I+\Op(e_{1s}),
\]
with $e_{1s} = e_{0sr}+\wte_{0sr}\in S^{-2,-2}$, given by
\begin{equation}\label{eq:e1s}
	e_{1s}(x,\xi)\sim \sum_{j\in\N}\frac{Q_{1,j}(x,\xi)}{[b_s(x,\xi)]^{j+3}},
\end{equation}
where the symbols $Q_{1,j}\in S^{(j+3)m-j-2,(j+3)\mu-j-2}$ are asymptotic sums of (derivatives of) 
polynomials $\wtP^{\theta\theta}_{j+3,1}$ of the form \eqref{eq:Pst}.
In fact,  by the second part of Lemma \ref{lem:deriv} and Leibniz formula, we find, for $s\in\C_\lambda$,
\begin{equation}\label{eq:wte0sae}
\begin{aligned}
	\wte_{0sr}(x,\xi)&\sim \phantom{-}\sum_{|\theta|>0} \frac{i^{|\theta|}}{\theta!}D^{\theta}_\xi b_s(x,\xi) D^\theta_x c_{1s}(x,\xi)
	\\
	&=-\sum_{|\theta|>0} \frac{i^{|\theta|}}{\theta!} D^\theta_\xi a(x,\xi)
	\left\{
		\sum_{\substack{\gamma\le\theta \\ \gamma\not=\theta}} \begin{pmatrix} \theta \\ \gamma \end{pmatrix}
		D^\gamma_x Q_{0,0}(x,\xi)\sum_{j=1}^{|\theta-\gamma|}\frac{P^{(\theta-\gamma)0}_j(x,\xi)}{[b_s(x,\xi)]^{j+3}}
		+ 
		\frac{D^\theta_x Q_{0,0}(x,\xi)}{[b_s(x,\xi)]^3}
	\right\}.
\end{aligned}
\end{equation}
We can rewrite, equivalently,
\begin{equation}\label{eq:wte0sr}
\begin{aligned}
	\wte_{0sr}(x,\xi) &\sim \frac{1}{[b_s(x,\xi)]^3}\sum_{|\theta|>0} \frac{-i^{|\theta|}}{\theta!} \, D_\xi^\theta a(x,\xi)D^\theta_x Q_{0,0}(x,\xi) 
	\\
	&+ \sum_{j\ge1}\frac{1}{[b_s(x,\xi)]^{j+3}}\sum_{|\theta|>j}\sum_{|\gamma|<|\theta|}C_{\theta\gamma} 
	D^\theta_\xi a(x,\xi) D^\gamma_x Q_{0,0}(x,\xi) P^{\theta-\gamma,0}_{j}(x,\xi).
\end{aligned}
\end{equation}
As in the previous step ii), we observe that the sums with respect to $|\theta|$ in \eqref{eq:wte0sr} are SG-asymptotic expansions in polynomials
$\wtP^{\wtth\wtth}_{j+3}$, $|\wtth|=|\theta|+1$, of the form \eqref{eq:Pst}, identifying symbols $\wtQ_{1,j}\in S^{(j+3)m-j-2,(j+3)\mu-j-2}$, $j\in\N$,
and giving $\wte_{0sr}\in S^{-2,-2}$, as claimed. Since, by the previous step,
\[
	e_{0sr}(x,\xi)\sim \sum_{j\in\N}\frac{Q_{0,j+1}(x,\xi)}{[b_s(x,\xi)]^{j+3}}, \quad Q_{0,j+1}\in S^{(j+3)m-j-2,(j+3)\mu-j-2}, j\in\N,
\]
we obtain \eqref{eq:e1s} setting $Q_{1,j}=\wtQ_{1,j}+Q_{0,j+1}$, $j\in\N$. Of course, by the previous step, also 
the symbols $Q_{1,j}$ admits expansions in polynomials $\wtP^{\wtth\wtth}_{j+3}$ of the form \eqref{eq:Pst}, for any $j\in\N$.\\

\noindent
iv) Set 
\[
	c_{2s}(x,\xi)=-\frac{Q_{1,0}(x,\xi)}{[b_s(x,\xi)]^4}.
\]
As in step iii), by the properties of $Q_{1,0}$ and hypoellipticity, it follows that $c_{2s}\in S^{-m^\prime-2,-\mu^\prime-2}$, with
seminorms $|||c_{2s}|||^{-m^\prime-2,-\mu^\prime-2}_k$, $k\in\N$, uniformly bounded with respect to $s\in\C_\lambda$. Writing
\[
	e_{1s}(x,\xi)\sim \frac{Q_{1,0}(x,\xi)}{[b_s(x,\xi)]^3}+\sum_{j\ge1}\frac{Q_{1,j}(x,\xi)}{[b_s(x,\xi)]^{j+3}}=e_{1sp}(x,\xi)+e_{1sr}(x,\xi),
\]
we find, similarly to step iii), $e_{1sp}\in S^{-2,-2}$ and $e_{1sr}\in S^{-3,-3}$. By the calculus,
\[
	\Op(b_s)\circ\Op(c_{0s}+c_{1s}+c_{2s})=I + \Op(e_{1sp}) + \Op(e_{1sr}) - \Op(e_{1sp}) + \Op(\wte_{1sr}) = I + \Op(e_{2s}),
\]
with $e_{2s} = e_{1sr}+\wte_{1sr}\in S^{-3,-3}$, given by
\begin{equation}\label{eq:e2s}
	e_{2s}(x,\xi)\sim \sum_{j\in\N}\frac{Q_{2,j}(x,\xi)}{[b_s(x,\xi)]^{j+4}},
\end{equation}
where the symbols $Q_{2,j}\in S^{(j+4)m-j-3,(j+4)\mu-j-3}$ are asymptotic sums of (derivatives of) 
polynomials $\wtP^{\theta\theta}_{j+4,1}$ of the form \eqref{eq:Pst}.
In fact, as in the previous step iii), we find
\begin{equation}\label{eq:wte1sae}
\begin{aligned}
	\wte_{1sr}(x,\xi)&\sim \phantom{-}\sum_{|\theta|>0} \frac{i^{|\theta|}}{\theta!}D^{\theta}_\xi b_s(x,\xi) D^\theta_x c_{2s}(x,\xi)
	\\
	&=-\sum_{|\theta|>0} \frac{i^{|\theta|}}{\theta!} D^\theta_\xi a(x,\xi)
	\left\{
		\sum_{\substack{\gamma\le\theta \\ \gamma\not=\theta}} \begin{pmatrix} \theta \\ \gamma \end{pmatrix}
		D^\gamma_x Q_{1,0}(x,\xi)\sum_{j=1}^{|\theta-\gamma|}\frac{P^{(\theta-\gamma)0}_j(x,\xi)}{[b_s(x,\xi)]^{j+4}}
		+ 
		\frac{D^\theta_x Q_{1,0}(x,\xi)}{[b_s(x,\xi)]^4}
	\right\}.
\end{aligned}
\end{equation}
We can rewrite, equivalently,
\begin{equation}\label{eq:wte1sr}
\begin{aligned}
	\wte_{1sr}(x,\xi) &\sim \frac{1}{[b_s(x,\xi)]^4}\sum_{|\theta|>0} \frac{-i^{|\theta|}}{\theta!} \, D_\xi^\theta a(x,\xi)D^\theta_x Q_{1,0}(x,\xi) 
	\\
	&+ \sum_{j\ge1}\frac{1}{[b_s(x,\xi)]^{j+4}}\sum_{|\theta|>j}\sum_{|\gamma|<|\theta|}C_{\theta\gamma} 
	D^\theta_\xi a(x,\xi) D^\gamma_x Q_{1,0}(x,\xi) P^{\theta-\gamma,0}_{j}(x,\xi).
\end{aligned}
\end{equation}
As above, we observe that the sums with respect to $|\theta|$ in \eqref{eq:wte1sr} are SG-asymptotic expansions in polynomials
$\wtP^{\wtth\wtth}_{j+4}$, $|\wtth|=|\theta|+2$,
of the form \eqref{eq:Pst}, identifying symbols $\wtQ_{2,j}\in S^{(j+4)m-j-3,(j+4)\mu-j-3}$, $j\in\N$, 
and giving $\wte_{1sr}\in S^{-3,-3}$, as claimed. Since, by the previous step,
\[
	e_{1sr}(x,\xi)\sim \sum_{j\in\N}\frac{Q_{1,j+1}(x,\xi)}{[b_s(x,\xi)]^{j+4}}, \quad Q_{1,j+1}\in S^{(j+4)m-j-3,(j+4)\mu-j-3}, j\in\N,
\]
we obtain \eqref{eq:e2s} setting $Q_{2,j}=\wtQ_{2,j}+Q_{1,j+1}$, $j\in\N$. Of course, by the previous step, also 
the symbols $Q_{2,j}$ admits expansions in polynomials $\wtP^{\wtth\wtth}_{j+4}$ of the form \eqref{eq:Pst}, for any $j\in\N$.\\

\noindent
v) Iterating step iv), we obtain remainders families 
\[
	e_{ks}\sim\sum_{j\in\N}\frac{Q_{k,j}}{b_s^{j+k+2}}, \quad
	Q_{k,j}\in S^{(j+k+2)m-j-k-1,(j+k+2)\mu-j-k-1}, j\in\N, k\ge 3,
\]
where the symbols $Q_{k,j}$ are independent of $s$ and obtained as asymptotic sums of derivatives of polynomials of the form \eqref{eq:Pst},
with monomials of degree $j+k+2$, and symbol families $c_{js}\in S^{-m^\prime-j,-\mu^\prime-j}$, $j\in\N$, such that:
\begin{itemize}
	\item[-] $c_{0s}=\dfrac{1}{b_s}$, $c_{js}=\dfrac{Q_{j-1}}{b_s^{j+2}}$, $Q_{j-1}=-Q_{j-1,0}\in S^{(j+1)m-j,(j+1)\mu-j}$, $j\ge1$;
	\item[-] $||| c_{js}|||^{-m^\prime-j,-\mu^\prime-j}_k$ is uniformly bounded with respect to $s\in \C_\lambda$, for any $j,k\in \N$;
	\item[-] for any $N\in\N$, $\Op(b_s)\Op(c_{0s}+c_{1s}+\cdots+c_{Ns})=I+\Op(e_{Ns})$, $e_{Ns}\in S^{-N-1,-N-1}$.
\end{itemize}
For any $s\in \C_\lambda$ we then consider the asymptotic sum $\displaystyle c_s=\sum_j c_{js}$,
providing a right parametrix to $\Op(b_s)$, namely,
\[
	\Op(b_s)\Op(c_s)=I+\Op(r_{1s}),
\]
for some $r_{1s}\in S^{-\infty,-\infty}$. By construction, it follows that for any $k\in \N$ there exists $C_k>0$ such that
$|||c_s|||^{-m^\prime,-\mu^\prime}_k\leq C_k$, uniformly with respect to $s\in \C_\lambda$. The construction of a left parametrix,
with analogous properties, follows by a completely similar argument. The proof is complete.
\end{proof}
\begin{cor}\label{cor:rs}
	The symbols $r_{1s}$ and $r_{2s}$ of the remainders in \eqref{eq:cs_par} satisfy 
	\begin{align*}
		\forall M\in\N\;\forall\alpha,\beta\in\N^d\; \exists C_{\alpha\beta}>0 \; \forall x,\xi\in\R^d\;\forall s\in\C_\lambda\;\;
		&
		|\partial^\alpha_\xi\partial^\beta_x r_{1s}(x,\xi)|\le |s|^{-rM}\csi^{-M-|\alpha|}\x^{-M-|\beta|}
		\\
		\text{ and }\;
		&|\partial^\alpha_\xi\partial^\beta_x r_{2s}(x,\xi)|\le |s|^{-rM}\csi^{-M-|\alpha|}\x^{-M-|\beta|},
	\end{align*}
	with $\lambda>0$ sufficiently large.
	We write $r_{1s},r_{2s} \in |s|^{-\infty}S^{-\infty,-\infty}(\R^d\times\R^d)$. It follows that the corresponding kernels $k_{1}(s,x,y)$ and $k_2(s,x,y)$
	of $\Op(r_{1s})$ and $\Op(r_{2s})$, respectively, satisfy analogous estimates in $(s,x,y)\in\C_\lambda\times\R^d\times\R^d$, and are analytic
	functions on $\C_\lambda$ taking values in $\scS(\R^d\times\R^d)$, with $\lambda>0$ sufficiently large.
\end{cor}
\begin{proof}
	By the proof of Proposition \ref{prop:cs_parametrix}, we already know that $r_{1s},r_{2s}\in S^{-\infty,-\infty}$, uniformly with respect to $s\in\C_\lambda$,
	$\lambda>0$ sufficiently large.
	It is then enough to check the fast decay property with respect to $|s|^r$ in the same complex domain. To this aim, we prove that
	\begin{enumerate}
		\item[ i)] for any $M\in\N$ it holds $\dfrac{s^{rM}}{b_s^M}\in S^{0,0}$, uniformly with respect to $s\in\C_\lambda$, $\lambda>0$ sufficiently large;
		\item[ii)] for any $M\in\N$ there exists $N\in\N$ such that $e_{Ns}\in |s|^{-rM}S^{-M,-M}$, uniformly with respect to $s\in\C_\lambda$, 
		$\lambda>0$ sufficiently large, with $e_{Ns}$ from point v) of the proof
		of Proposition \ref{prop:cs_parametrix}.
	\end{enumerate}
	
	Point i) follows immediately, observing that, for any $M\in\N$, $s\in\C_\lambda$, $x,\xi\in\R^d$, $\alpha,\beta\in\N^d$, $|\alpha+\beta|=1$,
	\begin{align*}
		&\left|\frac{s^{rM}}{[s^r+a(x,\xi)]^M}\right|\le\frac{|s|^{rM}}{|s^r|^M}
		=1,
		\\
		&\left|\partial^\alpha_\xi\partial^\beta_x\frac{s^{rM}}{[s^r+a(x,\xi)]^M}\right|
		=M\frac{|s|^{rM}}{|s^r+a(x,\xi)|^{M+1}}\cdot|\partial^\alpha_\xi\partial^\beta_x a(x,\xi)|
		\le M\frac{|\partial^\alpha_\xi\partial^\beta_x a(x,\xi)|}{a(x,\xi)}
		\lesssim \csi^{-|\alpha|}\x^{-|\beta|},
	\end{align*}
	and that the general estimates for arbitrary $\alpha,\beta\in\N^d$ can be obtained by induction, with constants not depending on $s\in\C_\lambda$.
	
	To prove point ii), set $N=\max\{2mM,2\mu M\}+2M$, so that
	\[
		s^{rM} e_{Ns} \sim \sum_{j\in\N}s^{rM}\frac{Q_{N,j}}{b_s^{j+N+2}} 
		= \sum_{j\in\N}\underbrace{\frac{s^{rM}}{b_s^M}}_{\in S^{0,0}}\cdot \underbrace{\frac{Q_{N,j}}{b_s^{j+N-M+2}}}_{\in S^{z,\zeta}},
	\]
	uniformly with respect to $s\in\C_\lambda$, where, by the choice of $N$, 
	\begin{align*}
		z &= (j+N+2)m-j-N-1-(j+N-M+2)m=Mm-j-N-1<-M-j,
		\\
		\zeta &= (j+N+2)\mu-j-N-1-(j+N-M+2)\mu=M\mu-j-N-1<-M-j,
	\end{align*}
	which implies $s^{rM}e_{sN}\in S^{-M,-M}\Rightarrow e_{sN}\in |s|^{-rM}S^{-M,-M}$, $s\in\C_\lambda$, $\lambda>0$ sufficiently large.
	The remaining claims are immediate, in view of the properties of the kernels of smoothing operators in the SG-calculus and of the sums
	of asymptotic expansions.
\end{proof}

\section{Laplace transform of functions and distributions}\label{subs:LTandpsidos}
\setcounter{equation}{0}
Here we recall the basic definitions and properties of the (vector-valued) Laplace transform on functions and distributions, 
to fully justify our approach and the functional setting where we looked for the solutions of the Cauchy problems we studied in \cite{CGP1}. For the sake of completeness,
we also prove some properties that we tacitly employed in some of the proofs in \cite{CGP1}, concerning the commutation properties of 
the (inverse) Laplace transform with respect to the $t$-variable ($s$-variable) and the action of pseudodifferential operators on (families of) temperate 
distributions in the $x$-variable.

\medskip

The next definitions and results are well-known. Here we mainly follow \cite{Gilardi} (cf. also \cite{Doetsch,AMP}), and describe the results
for $\C$-valued functions and corresponding distributions. The extension to vector-valued functions and distributions is rather straightforward, 
and we only comment shortly about this in Remark \ref{rem:LTvector}. 

\begin{defn}\label{def:LTfun}
	A function $u\in L^1_{\loc}(\R)$ is called ($\cL$-)transformable, and we write $u\in \funLTR$, if
	\begin{enumerate}[label={($\cL\mathrm{T}_\arabic*$)},ref={($\cL\mathrm{T}_\arabic*$)}]
		\item\label{LT:supp} $\supp u\subseteq[0,+\infty)$ and
		\item\label{LT:expfun} there exists $\lambda\in\R$ such that $t\mapsto e^{-\lambda t}u(t)\in L^1(\R)$.
	\end{enumerate}
	The number $\absa(u)=\inf\{\lambda\in\R\colon \text{condition \ref{LT:expfun} holds true}\}$ is called \textit{(absolute) abscissa of convergence}
	of the Laplace integral of $u$. $\funLTR$ is a vector space.
	
	The Laplace transform of $u\in\funLTR$ is the function defined by
	\begin{equation}\label{eq:LTufun}
		(\cL u)(s)=\int_{-\infty}^{+\infty} e^{-st} u(t)\, dt,
	\end{equation}
	for any $s\in\C$ such that $\Re s>\absa(u)$. In the sequel we will often employ the notation $e_\tau(t)=\exp(-\tau t)$.
\end{defn}
\begin{rem}\label{rem:LTfun}
	We immediately see that, if $u\in \funLTR$ with $\lambda\in\R$ satisfying \ref{LT:expfun}, $\Re s > \lambda$ implies
	$e_s u\in L^1(\R)$. This is of course the case if $\Re s>\absa(u)$. In such situation, $(\cL u)(s)=\widehat{e_{\Re s} u}(\Im s)$.
	Notice that it can happen that $\absa(u)=-\infty$, so that the condition $\Re s>\absa(u)$ is void. 
	The set $\{s\in\C\colon\Re s>\absa(u)\}=\C_{\absa(u)}$ is anyway usually called \textit{half-plane of absolute convergence} 
	(even in the exceptional case when it actually is the full complex plane).
\end{rem}
\begin{prop}\label{prop:LTfunprop}
	Let $u\in\funLTR$. Then it holds:
	\begin{enumerate}
		\item $\cL u$ is bounded on the half-plane $\overline{\C_\lambda}$ for any $\lambda>\absa(u)$;
		\item $\displaystyle\lim_{\Re s\to+\infty}(\cL u)(s) = 0\Leftrightarrow \forall (s_n)\text{ such that }
		\lim_{n\to+\infty}\Re s_n=+\infty \colon \lim_{n\to+\infty}(\cL u)(s_n)=0$.
	\end{enumerate}
\end{prop}
\begin{thm}\label{thm:LTfunthma}
	Let $u\in\funLTR$. Then:
	\begin{enumerate}
		\item $\cL u$ is holomorphic in $\C_{\absa(u)}$;
		\item for any $k\in \N$ the function $v_k(t)= t^ku(t)$ belongs to $\funLTR$ and its abscissa of absolute convergence is $\absa(u)$;
		\item for any $k\in \N$ it holds $\displaystyle\frac{d^k}{ds^k}(\cL u)(s)=(-1)^k(\cL v_k)(s)$, $\Re s>\absa(u)$.
	\end{enumerate}
\end{thm}
\begin{thm}\label{thm:LTfunthmb}
	Let $u,v\in\funLTR$. Then:
	\begin{enumerate}
		\item the following functions are in $\funLTR$, too, and the corresponding formulae hold true:
			\begin{enumerate}
				\item $\displaystyle w(t)=u(ct)=(M_c^*u)(t), c\not=0 \Rightarrow (\cL w)(s)=\frac{1}{c}(\cL u)\left(\frac{s}{c}\right)$, $\Re s> c\absa(u)$;
				\item $w(t)=u(t-t_0)=(\tau_{t_0}u)(t), t_0>0 \Rightarrow (\cL w)(s)=e^{-t_0 s} (\cL u)(s)$, $\Re s > \absa(u)$;
				\item $w(t)=e^{s_0 t} u(t)=(e_{-s_0}\cdot u)(t), s_0\in\C \Rightarrow (\cL w)(s)=(\cL u)(s-s_0)$, $\Re s> \absa(u)+\Re s_0$;
			\end{enumerate}
		\item if $u^\prime\in\funLTR$, then $(\cL u^\prime)(s) = s(\cL u)(s)$, $\Re s>\max\{\absa(u),\absa(u^\prime)\}$;
		\item if the function $w(t)=\dfrac{u(t)}{t}$ belongs to $\funLTR$,
		then $\displaystyle (\cL w)(s)=\int_s^\infty (\cL u)(\tau)\,d\tau$, $\Re s>\absa(u)$, where the integral can
		be possibly understood in improper sense;
		\item $u*v\in\funLTR$ and $(\cL(u*v))(s)=(\cL u)(s)\cdot (\cL v)(s)$, $\Re s > \max\{\absa(u),\absa(v)\}$.
	\end{enumerate}
\end{thm}
The definition of $\cL$-transformable distributions is given in analogy to Definition \ref{def:LTfun}, substituting $L^1(\R)$ with $\scS^\prime(\R)$.
Notice that, when $u$ has compact support, the integral \eqref{eq:LTufun} can be interpreted as the action of $u$ as a distribution on the family 
of functions $e_s$. This, together with the hypotheses on the support of $u$, leads to the following definition of Laplace transform of 
$\cL$-transformable distributions.
\begin{defn}\label{def:LTdis}
	A distribution $u\in\scD^\prime(\R)$ is called ($\cL$-)transformable, and we write $u\in\disLTR$, if it satifies
	\ref{LT:supp} and
	\begin{enumerate}[label={($\cL\mathrm{T}_\arabic*$)},ref={($\cL\mathrm{T}_\arabic*$)}]
		\setcounter{enumi}{2}
		\item\label{LT:expdis} there exists $\lambda\in\R$ such that $e_\lambda u\in \scSp(\R)$.
	\end{enumerate}
	The number $\lambda(u)=\inf\{\lambda\in\R\colon \text{condition \ref{LT:expdis} holds true}\}$ is called \textit{abscissa of convergence}
	of the Laplace integral of $u$. $\disLTR$ is a vector space. 
	
	Let $\zeta\in C^\infty(\R)$ satisfy, for some $a>0$
	\begin{equation}\label{eq:zetacond}
		\zeta(t)=0 \text{ for $t\in(-\infty,-a]$ and } \zeta(t)=1 \text{ for $t\in[-a/2,+\infty)$}. 
	\end{equation}
	The Laplace transform of $u\in\disLTR$ is the function defined by
	\begin{equation}\label{eq:LTudis}
		(\cL u)(s)=(e_\lambda u)(e_{s-\lambda}\zeta),
	\end{equation}
	for any $s\in\C$ such that $\Re s>\lambda(u)$, with $\lambda\in(\lambda(u),\Re s)$.
\end{defn}
\begin{rem}\label{rem:LTdis}
	The definition \eqref{eq:LTudis} makes sense, since, under the hypotheses on $u,s,\lambda$, and $\zeta$, $e_\lambda u\in\scSp(\R)$
	and $e_{s-\lambda}\zeta\in\scS(\R)$. Moreover, if $u\in\funLTR$, for $s,\lambda$, and $\zeta$ as in Definition \ref{def:LTdis}, we find
	\[
		\int_{-\infty}^{+\infty} e^{-st} u(t)\,dt = 
		\int_{-\infty}^{+\infty} e^{-\lambda t} e^{-(s-\lambda)t}\zeta(t)\,u(t)\,dt = 		
		\int_{-\infty}^{+\infty} e^{-\lambda t} u(t) \, e^{-(s-\lambda)t}\zeta(t)\,dt=(e_\lambda u)(e_{s-\lambda}\zeta),
	\]
	so Definition \ref{def:LTdis} is consistent with Definition \ref{def:LTfun}, and clearly $\cL u$, defined on $\C_{\absa(u)}$,
	does not depend on $\lambda$, $\zeta$, $a>0$, satisfying the hypotheses stated in Definition \ref{def:LTdis} when $u\in\funLTR$. 
	The same holds true for any $u\in\disLTR$ on $\C_{\lambda(u)}$. 
	Moreover,
	$u\in\scE^	\prime(\R)$ implies $u\in\disLTR$ with $\lambda(u)=-\infty$ and, in such case, one actually has $(\cL u)(s)=u(e_s)$, $s\in\C$.
\end{rem}
\begin{defn}\label{def:succLTdis}
	A sequence $(u_n)_{n\in\N}\subset\disLTR$ converges to $u\in\disLTR$ (or \textit{in $\disLTR$} or \textit{in the sense of ($\cL$)-transformable distributions})
	if there exists $\lambda\in\R$ such that 
	\begin{equation}\label{eq:convLTRdis}
		e_\lambda u_n \to e_\lambda u, \; n\to+\infty, \text{ in } \scSp(\R).
	\end{equation}
	Similarly, a series of distributions in $\disLTR$ converges to $u\in\disLTR$ if this holds true for the associated sequence of partial sums.
\end{defn}
\begin{thm}\label{thm:succLTdis}
	If a sequence $(u_n)_{n\in\N}\subset\disLTR$ converges to $u$ in $\disLTR$, then there exists $\lambda\in\R$ such that $\lambda(u_n)\le\lambda$
	for any $n\in\N$, $\lambda(u)\le\lambda$, and $(\cL u_n)_{n\in\N}$ is pointwise convergent to $\cL u$ in $\C_\lambda$. More precisely, such claims
	hold true if $\lambda$ satisfies \eqref{eq:convLTRdis}.
\end{thm}
\begin{cor}\label{cor:succLTdis}
	Let $(a_n)_{n\in\N}\subset\C$ and consider the series
	\begin{equation}\label{eq:exmplseries}
		\sum_{n=0}^\infty a_n t^n H(t)
	\end{equation}
	and the corresponding series of term-by-term Laplace transforms,
	\begin{equation}\label{eq:exmplseriesLT}
		\sum_{n=0}^\infty \frac{a_n n!}{s^{n+1}}.
	\end{equation} 
	Assume that $r>0$ is such that \eqref{eq:exmplseriesLT} converges for $|s|>r$. Then the series \eqref{eq:exmplseries} converges
	pointwise and in $\disLTR$ to $u\in\disLTR$ and $(\cL u)(s)$ is given by the sum of \eqref{eq:exmplseriesLT} on $\C_r$.
\end{cor}
The following approximation result will be useful in the sequel.
\begin{thm}\label{thm:apprdis}
	Let $u\in\disLTR$. Then there exists a sequence $(u_n)_{n\in\N}\subset\scD(\R)\cap\funLTR$ such that $(u_n)_{n\in\N}$
	converges to $u$ in $\disLTR$. More precisely, the sequence can be chosen such that $e_\lambda u_n \to e_\lambda u$
	in $\scSp(\R)$ for any $\lambda>\lambda(u)$.
\end{thm}
The next two Theorems \ref{thm:LTfunthma} and \ref{thm:LTfunthmb} are the analog in $\disLT$ of Theorems \ref{thm:LTfunthma} and \ref{thm:LTfunthmb}
in $\funLT$.
\begin{thm}\label{thm:LTdisthma}
	Let $u\in\disLTR$. Then:
	\begin{enumerate}
		\item $\cL u$ is holomorphic in $\C_{\lambda(u)}$;
		\item the distribution $v= \mathrm{id}\cdot u$, $\mathrm{id}(t)=t$, belongs to $\disLTR$ and its abscissa of convergence is $\lambda(u)$;
		\item it holds $\displaystyle\frac{d}{ds}(\cL u)(s)=-(\cL v)(s)$, $\Re s>\lambda(u)$.
	\end{enumerate}
\end{thm}
\begin{thm}\label{thm:LTdisthmb}
	Let $u,v\in\disLTR$. Then:
	\begin{enumerate}
		\item the following distributions are in $\disLTR$, too, and the corresponding formulae hold true:
			\begin{enumerate}
				\item $\displaystyle w=M_c^*u, M_c\colon t\mapsto c\cdot t, c\not=0 \Rightarrow (\cL w)(s)=\frac{1}{c}(\cL u)\left(\frac{s}{c}\right)$, $\Re s> c\lambda(u)$;
				\item $w=\tau_{t_0}u, t_0>0 \Rightarrow (\cL w)(s)=e^{-t_0 s} (\cL u)(s)$, $\Re s > \lambda(u)$;
				\item $w=e_{-s_0}\cdot u, s_0\in\C \Rightarrow (\cL w)(s)=(\cL u)(s-s_0)$, $\Re s> \lambda(u)+\Re s_0$;
			\end{enumerate}
		\item $u\in\disLTR\Leftrightarrow u^\prime\in\disLTR$, $\lambda(u^\prime)\le\lambda(u)$, and $(\cL u^\prime)(s) = s(\cL u)(s)$, $\Re s>\lambda(u)$;
		%
%		\item if the function $w(t)=\dfrac{u(t)}{t}$ belongs to $\funLTR$,
%		then $\displaystyle (\cL w)(s)=\int_s^\infty (\cL u)(\tau)\,d\tau$, $\Re s>\absa(u)$, where the integral can
%		be possibly understood in improper sense;
%		%
		\item $u*v\in\disLTR$ and $(\cL(u*v))(s)=(\cL u)(s)\cdot (\cL v)(s)$, $\Re s > \max\{\lambda(u),\lambda(v)\}$.
	\end{enumerate}
\end{thm}
To employ the Laplace transform to solve Initial Value Problems associated with (partial) differential equations, 
one needs, on one hand, to handle initial values, on the other hand, to invert $\cL$, 
similarly to what happens with the Fourier transform $\scF$. 

\medskip

Concerning initial values, recall the following (distributional) identity for the functions $f,g\in C([0,\infty))$:
\begin{equation*}
	\partial_t(f*g)(t)=f(t)g(0)+(f*\partial_t g)(t), \;\mbox{ where } (f*g)(t)=\int_0^tf(t-s)g(s)\,ds, \ t\in[0,\infty).
\end{equation*}
Let $(\theta_\nu)_{\nu\in\N}\subset \scD(\R)$ be a sequence of test functions of the form $\theta_\nu(t)=\nu\theta(\nu t), \;  t\in\R,  \nu\in\N$,
where $\theta\in \scD(\R)$, $\supp\theta\subset[0,a]$ for some $a>0$, and
$ \int_{\mathbb R}\theta(t)dt=1$.  This is a delta sequence, that is, $\theta_\nu\rightarrow \delta$, $\nu\rightarrow \infty$, 
in $\mathcal S'(\R)$ and in $\mathcal S'([0,+\infty))$.
If $f$ is a derivative of order $k\in\N$ of some exponentially bounded continuous function supported in $[0,\infty)$, then
its Laplace transform is given by
\begin{equation}\label{LAP}
(\cL f)(s)=\lim_{\nu\rightarrow \infty} \cL(f*\theta_\nu)(s),\; \Re s>\lambda,
\end{equation}
if this limit exists for $ \Re s>\lambda$.
So, choosing $f=\dfrac{dF}{dt}$ and assuming that it is an exponentially bounded continuous function, \eqref{LAP} gives
\[
	\cL\left(\frac{d F}{dt}\right)(s)=s(\cL F)(s)-F(0),\quad \Re s>\lambda, \lambda\geq 0.
\]
Let  $F\in C([0,\infty))$ be exponentially bounded. There holds
\[
	\langle\partial_t(\theta_\nu*F)(t),e^{-st}\rangle=\langle \theta_\nu(t)F(0),e^{-st}\rangle+\langle (\theta_\nu*\partial_t F)(t),e^{-st}\rangle,
\]
which implies
\[
	s(\cL F)(s)=F(0)+(\cL f)(s), \quad \Re s>\lambda.
\]

\medskip

Concerning the inversion of $\cL$, it is necessary to characterise the functions $f$ of one complex variable
which are Laplace transforms of a distribution $u\in\disLT$ (or of a function $u\in\funLT$), and to identify $u$ by means of $f$. A first step
towards the inversion formula is given by Remarks \ref{rem:LTfun} and \ref{rem:LTdis}, concerning the relationship between Laplace and Fourier transforms.
Explicitly, in the case of $u\in\funLT$, setting $x=\Re s>\absa(u)$ and $y=\Im s$, we have that $f\colon y\mapsto (\cL u)(x+iy)$ is the Fourier transform
of the $L^1$ function $t\mapsto (e_x u)(t)$, and, conversely, the latter is the inverse Fourier transform of the former. Proceeding formally (this is
correct if $f\in L^1(\R_y)$),
\begin{align*}
	u(t)=e^{xt}\cdot e^{-xt}u(t)&=e^{xt}\cdot\frac{1}{2\pi}\int_{-\infty}^{+\infty}e^{iyt}f(x+iy)\,dy=\frac{1}{2\pi}\int_{-\infty}^{+\infty}e^{(x+iy)t}f(x+iy)\,dy
	\\
	&=\frac{1}{2\pi i}\lim_{R\to+\infty}\int_{-R}^{+R}e^{(x+iy)t}f(x+iy)\,i\,dy.
\end{align*}
Interpreting the last integral as a path-integral, one is led to
\[
	u(t)=\frac{1}{2\pi i}\lim_{R\to+\infty}\int_{[x-iR,x+iR]}e^{st}f(s)\,ds,
\]
or, as commonly written,
\begin{equation}\label{eq:RF}
	u(t)=\frac{1}{2\pi i}\int_{x-i\infty}^{x+i\infty}e^{st}f(s)\,ds.
\end{equation}
\eqref{eq:RF} is the so-called Riemann-Fourier formula. It holds true if $x>\absa(u)$ and $u$ satisfies some conditions,
precisely stated in Theorem \ref{thm:LTinv} below.

\begin{thm}\label{thm:LTinv}
	Let $f$ be a function of one complex variable. Then $f$ is the Laplace transform of a distribution $u\in\disLTR$ if and only if
	there exists $\lambda\in\R$ such that
	\begin{enumerate}[label={($\cL\mathrm{T}^{-1}_\arabic*$)},ref={($\cL\mathrm{T}^{-1}_\arabic*$)}]
		\item\label{point:ltinva} $f$ is holomorphic on the half-plane $\C_\lambda$;
		\item\label{point:ltinvb} $\exists M,m\colon |f(s)|\le M(1+|s|)^m$ for $s\in\C_\lambda$.
	\end{enumerate}
	If the conditions \ref{point:ltinva} and \ref{point:ltinvb} are satisfied, the inverse Laplace transform $\cL^{-1}f$ is uniquely
	determined by $f$ and satisfies
	\begin{equation}\label{eq:Ltinvchar}
		\cL^{-1}f = \left[\cL^{-1}_{s\to \cdot}\left(\frac{f(s)}{s^n}\right)\right]^{(n)},\quad n\in\N.
	\end{equation}
	Moreover, if
	\begin{equation}\label{eq:Ltinvreg}
		\exists \mu\ge \lambda\;\exists \alpha>1 \;\exists M^\prime\colon |f(s)|\le M^\prime|s|^{-\alpha},s\in\C_\mu,
	\end{equation}
	$\cL^{-1}f$ coincides with the function $u$ given by the Riemann-Fourier formula \eqref{eq:RF}, with arbitrary $x>\mu$,
	and $u$ is a continuous function on $\R$.
\end{thm}
\begin{rem}\label{rem:LTvector}
	Let $u\colon\R\to E$, where $E$ is a Frech\'et space. Definition \ref{def:LTfun} extends to this more general situation,
	and produces a function $\cL  u$ which is holomorphic and takes values in $E$. In a similar fashion, we can consider
	the Laplace transform of distributions taking values in the dual space $E^\prime$. The results above then extends, with
	straightforward modifications, to the spaces $\funLT(\R,\scS(\R^d))$, $\funLT(\R,\scSp(\R^d))$, $\disLT(\R,\scS(\R^d))$, and
	$\disLT(\R,\scSp(\R^d))$, 
\end{rem}
Let us now focus on the interplay between the Laplace transform of transformable distributions taking values in $\scSp$ and
pseudodifferential operators. The next Theorem \ref{thm:psidosLT} is commonly accepted, and we employed it in \cite{CGP1}.
However, since we could not find it stated or proved in the literature we could access, for the sake of completeness,
we give here a proof.
\begin{thm}\label{thm:psidosLT}
	Let $u\in\disLT(\R,\scSp(\R^d))$ and $a\in S^\mu(\R^d\times\R^d)$ or $a\in S^{m,\mu}(\R^d\times\R^d)$, $m,\mu\in\R$.
	Then, $\Op(a)u\in\disLT(\R,\scSp(\R^d))$, $\lambda(\Op(a)u)\le\lambda(u)$, and $\cL[\Op(a)u]=\Op(a)(\cL u)$ on $\C_{\lambda(u)}$.
\end{thm}
\begin{lem}\label{lem:LTtenspr}
	Let $u=v\otimes w$, where $v\in\disLT(\R)$ and $w\in\scSp(\R^d)$. Then, $u\in\disLT(\R,\scSp(\R^d))$ and, for 
	$a\in S^\mu(\R^d\times\R^d)$ or $a\in S^{m,\mu}(\R^d\times\R^d)$, $m,\mu\in\R$, 
	$\Op(a)u\in\disLT(\R,\scSp(\R^d))$, $\lambda(\Op(a)u)\le\lambda(u)$, and $\cL[\Op(a)u]=\Op(a)(\cL u)$ on $\C_{\lambda(u)}$.
\end{lem}
\begin{proof}
	By definition, there exists $\lambda(v)$ such that for $\lambda>\lambda(v)$, $e_\lambda v\in\scSp(\R)$. 
	By nuclearity of $\scSp$, for $\lambda>\lambda(u):=\lambda(v)$,
	\[
		e_\lambda u =	(e_\lambda v)\otimes w\in\scSp(\R)\otimes\scSp(\R^d)\hookrightarrow
		\scSp(\R)\widehat{\otimes}\scSp(\R^d)\simeq\scSp(\R^{1+d})\simeq\scSp(\R,\scSp(\R^d)),
	\]
	which proves $u\in\disLT(\R,\scSp(\R^d))$. Then, under the stated hypothesis on the symbol $a$, for any 
	$\lambda>\lambda(u)$,
	\begin{align*}
		\Op(a)(e_\lambda u)=e_\lambda \Op(a)u&\equiv (\mathrm{id}\otimes\Op(a))(e_\lambda u)=(e_\lambda v)\otimes\Op(a)w
		\\
		&=e_\lambda(v\otimes\Op(a)w)\in\scSp(\R)\otimes\scSp(\R^d)\hookrightarrow\scSp(\R,\scSp(\R^d)).
	\end{align*}
	It follows that, as distributions in $\scD^\prime(\R,\scSp(\R^d))$, $\Op(a)u=v\otimes\Op(a)w$, and, for $\lambda>\lambda(u)$,
	$e_\lambda\Op(a)u=\Op(a)(e_\lambda u)\in\scSp(\R,\scSp(\R^d))$,
	which shows $\Op(a)u\in\disLT(\R,\scSp(\R^d))$ and $\lambda(\Op(a)u)\le\lambda(u)$.
	With $\varphi\in\scS(\R^d)$, $s\in\C$, $\lambda\in\R$ satisfying $\Re s>\lambda>\lambda(u)$, and $\zeta\in C^\infty(\R)$ as in Definition \ref{def:LTdis},
	we compute
	\begin{align*}
		[(\cL u)(s)](\varphi)&= [(e_\lambda u)(\zeta e_{s-\lambda})](\varphi) = ((e_\lambda v)\otimes w)((\zeta e_{s-\lambda})\otimes\varphi)
		= (e_\lambda v)(\zeta e_{s-\lambda})\cdot w(\varphi)=[(\cL v)(s)]\cdot w(\varphi)
		\\
		&\Leftrightarrow (\cL u)(s)=(\cL v)(s)\cdot w.
	\end{align*}
	Then, for $s\in\C$, $\lambda\in\R$ satisfying $\Re s>\lambda>\lambda(u)$, and $\zeta$ as above,
	\begin{align*}
		\Op(a)((\cL u)(s)) &= (\cL v)(s)\cdot\Op(a)w=(e_\lambda v)(\zeta e_{s-\lambda})\cdot\Op(a) w
		=[\cL(v\otimes\Op(a)w)](s)
		\\
		&= [\cL(\mathrm{id}\otimes\Op(a))(v\otimes w)](s)=[\cL(\Op(a)u)](s).
	\end{align*}
	The proof is complete.
\end{proof}
\begin{proof}[Proof of Theorem \ref{thm:psidosLT}]
	Notice that, in particular, $u\in\scD^\prime(\R,\scSp(\R^d))$. By nuclearity, $\scD^\prime(\R,\scSp(\R^d))\simeq\scD^\prime(\R)\widehat{\otimes}\scSp(\R^d)$,
	so that
	\[
		u=\sum_{j\in\N} u_j\,v_j\otimes w_j, \quad (u_j)_{j\in\N}\in\ell^1, \quad v_j\in\scD^\prime(\R),w_j\in\scSp(\R^d),j\in\N.
	\]
	Let $\lambda>\lambda(u)$. Then, by hypothesis, $e_\lambda u\in\scSp(\R,\scSp(\R^d))$, and, again by nuclearity,
	$\scSp(\R,\scSp(\R^d))\simeq\scSp(\R^{1+d})\simeq\scSp(\R)\widehat{\otimes}\scSp(\R^d)$, so that
	\begin{align*}
		e_\lambda u &=\sum_{j\in\N} u_j (e_\lambda v_j)\otimes w_j\in\scSp(\R)\widehat{\otimes}\scSp(\R^d)
		\\
		&\Rightarrow e_\lambda v_j\in\scSp(\R), j\in\N\Rightarrow v_j\in\disLT(\R) \mbox{ and } \lambda(v_j)\le\lambda(u), j\in\N.
	\end{align*}
	Consider now, for any $N\in\N$, the finite sums $\displaystyle u_N=\sum_{j=0}^N u_j \,v_j\otimes w_j$. For $\Re s>\lambda(u)$ and any $N\in\N$,
	by Lemma \ref{lem:LTtenspr}, we find
	\[
		(\cL u_N)(s)=\sum_{j=0}^N u_j (\cL v_j)(s)\cdot w_j, \qquad
		\Op(a)u_N=\sum_{j=0}^N u_j\,  v_j\otimes \Op(a)w_j,
	\]
	and
	\[
		[\cL(\Op(a)u_N)](s)=\sum_{j=0}^N u_j (\cL v_j)(s)\cdot \Op(a)w_j
		= \sum_{j=0}^N \Op(a)(u_j (\cL v_j)(s)\cdot w_j)  =\Op(a)((\cL u_N)(s)).
	\]
	Notice that $u_N\to u$, $N\to\infty$, in $\disLT(\R,\scSp(\R^d))$, since $e_\lambda u_N\to e_\lambda u$, $N\to\infty$, in $\scSp(\R,\scSp(\R^d))$,
	for any $\lambda>\lambda(u)$. Then, by the analogue of Theorem \ref{thm:succLTdis} in this setting, $\cL u_N\to\cL u$, $N\to\infty$, pointwise in
	$\cA(\C_{\lambda(u)},\scSp(\R^d))$ and $(\cL u_N)(s)\to(\cL u)(s)$, $N\to\infty$, in $\scSp(\R^d)$, $s\in\C_{\lambda(u)}$, which implies, by continuity
	of $\Op(a)$ on $\scSp(\R^d)$, 
	\[
		\Op(a)((\cL u_N)(s))\to \Op(a)((\cL u)(s)), \quad N\to\infty, s\in\C_{\lambda(u)}.
	\]
	Moreover, for $\psi$ as in \eqref{eq:LTudis}, $\Re s >\lambda>\lambda(u)$,
	\begin{align*}
		[\cL(\Op(a)u_N)](s)&=(e_\lambda\Op(a)u_N)(\zeta e_{s-\lambda})=[\Op(a)(e_\lambda u_N)](\zeta e_{s-\lambda})
		=[(\mathrm{id}\otimes\Op(a))(e_\lambda u_N)](\zeta e_{s-\lambda})
		\\
		&\to (\mathrm{id}\otimes\Op(a))(e_\lambda u)](\zeta e_{s-\lambda})=(e_\lambda\Op(a)u)(\zeta e_{s-\lambda})=[\cL(\Op(a)u)](s), \quad N\to\infty.
	\end{align*}
	The last claim follows from the fact that $\mathrm{id}\otimes\Op(a)=\Op(b)$, for an amplitude
	$b\in\cA^{\max\{\mu,0\}}(\R^d\times\R^d)$ or $b\in\cA^{\max\{m,0\}+\max\{\mu,0\}}(\R^d\times\R^d)$,
	respectively, and such operators $\Op(b)$ linearly and continuously map $\scS(\R^{1+d})$ into itself, 
	as well as $\scSp(\R^{1+d})$ into itself. The proof is complete.
\end{proof}

\begin{rem}\label{rem:LTinv}
	Results similar to Theorem \ref{thm:psidosLT} hold true for the inverse Laplace transform, even
	when the symbol $a$ depends also on the complex variable $s\in\C_\lambda$. In the latter case, convolution-like
	superpositions of actions of pseudodifferential operators appear (as in the proof of 
	Theorem 3 in \cite{CGP1}). %\ref{thm:main_nhom}). 
	Details are left for the reader.
\end{rem}

\medskip

We conclude this section with some results which relate the decay properties of 
a transformable function with those of the corresponding Laplace transform.
\begin{lem}
\label{lem:Watson_0}
Let $\psi\in L^1_{\loc}([0,+\infty))$ satisfy the asymptotic property
\[
	\psi(t)\sim B t^\sigma \quad \text{as} \quad t \to +\infty,
\]
for some constants $B\in \C$ and $\sigma\in \R$ with $\sigma>-1$. Then $\psi\in \funLTR$ and $\cL\psi$ satisfies
\[
	(\cL\psi)(s)\sim B \frac{\Gamma(\sigma+1)}{s^{\sigma+1}},
\]
as $s\to0$ within the angular region $|\arg(s)|\leq \tilde\theta<\pi/2$.
\end{lem}
\begin{proof}
See the proof of Theorem 34.1 in \cite{Doetsch}.
\end{proof}
\begin{lem}
\label{lem:Watson_0<-1}
Let $\psi\in\funLTR$ satisfy $\lambda_a(\psi)=0$ and the asymptotic property
\[
	\psi(t)\sim B t^\sigma \quad \text{as} \quad t \to +\infty,
\]
for some constants $B\in \C$ and  $\sigma\in \R$ with $\sigma<-1$. Then there exists $C>0$ such that 
\[|(\cL\psi)(s)|\leq C,\]
uniformly with respect to $s$ in the angular region $|\arg(s)|\leq \tilde\theta<\pi/2$.
\end{lem}
\begin{proof}
Since $\psi(t)\sim Bt^\sigma$ as $t\to+\infty$, for any $\delta>0$ there exist $M>0$ sufficiently large such that $|\psi(t)-Bt^\sigma|\leq \delta t^\sigma$ for any 
$t\geq M$. Moreover, it holds
\[
\int_0^\infty e^{-st}\psi(t)\,dt - B\int_M^\infty e^{-st }t^\sigma \,dt 
= \int_0^M e^{-st}\psi(t)\,dt+\int_M^\infty e^{-st} (\psi(t)-Bt^\sigma)\,dt.
\]
On the one hand, there exists $K>0$ such that
\[
	\left|\int_0^M e^{-st}\psi(t)\,dt\right|\leq \int_0^M e^{-\Re(s)t}|\psi(t)|\,dt \leq K,
\]
since $\psi\in L^1_{\loc}([0,+\infty))$. On the other hand, we may estimate
\[
\left|\int_M^\infty e^{-st} (\psi(t)-Bt^\sigma)\,dt\right|\leq \delta \int_M^\infty e^{-\Re(s)t}t^\sigma \,dt\leq \frac{\delta M^{\sigma+1}}{\sigma+1},
\]
uniformly with respect to $s\in \C_+:=\{z\in \C: \Re(z)>0\}$. Similarly, one can estimate 
\[
	B\int_M^\infty e^{-st }t^\sigma \,dt \leq \frac{BM^{\sigma+1}}{\sigma+1}.
\]
The proof of the desired result follows, taking $C=K-(\delta+B)M^{\sigma+1}/(\sigma+1)>0$.
\end{proof}
\begin{lem}
\label{lem:Watson_infty}
Let $\psi\in\funLTR$ satisfy $\lambda_a(\psi)=0$. Assume also that 
\[
	\psi(t)\sim A t^\delta, \quad \text{as} \quad t \to 0^+,
\]
for some constants $A\in \C$ and $\delta\in \R$ with $\delta>-1$. Then it holds 
\[
	(\cL\psi)(s)\sim A \frac{\Gamma(\delta+1)}{s^{\delta+1}},
\]
as $s\to\infty$ in the angular region $|\arg(s)|\leq \theta<\pi/2$.
\end{lem}
\begin{proof}
See the proof of Theorem 33.3 in \cite{Doetsch} (see also Lemma 2.2 in \cite{DAG2023}).
\end{proof}

\section*{Declarations}

\noindent
\textbf{Data.} No data have been employed in this research.

\noindent
\textbf{Conflicts of interest/Competing interests.} The authors have no relevant financial or non-financial interests to disclose.

\noindent
\textbf{Funding.} The first author has been partially supported by the Italian Ministry of the University and Research - MUR, within the framework of the Call relating to the scrolling of the final rankings of the PRIN 2022 - Project Code 2022HCLAZ8, CUP D53C24003370006 (PI A. Palmieri, Local unit Sc. Resp. S. Coriasco). The first author also expresses
gratitude for the hospitality extended to him during his visit to the Department of Mathematics and Informatics, University of Novi Sad, Serbia, during A.Y. 2024/2025,
where part of this work was developed. The second author has been partially supported by INdAM GNAMPA Project, Grant Code CUP E55F22000270001.
The third author has been supported by the Serbian Academy of Sciences and Arts, project F10.

\noindent
\textbf{Authors' contribution.} All authors contributed to the study conception and design. All authors commented on previous versions of the manuscript. All authors read and approved the final manuscript.

\end{document}